\documentclass[12pt]{amsart}
\usepackage{amsmath}
\usepackage{geometry,amsfonts,amssymb,amsthm,txfonts,pxfonts,amscd} 
\def\struckint{\mathop{%
\def\mathpalette##1##2{\mathchoice{##1\displaystyle##2}%
  {##1\textstyle##2}{##1\scriptstyle##2}{##1\scriptscriptstyle##2}}%
\mathpalette
{\vbox\bgroup\baselineskip0pt\lineskiplimit-1000pt\lineskip-1000pt
\halign\bgroup\hfill$}
{##$\hfill\cr{\intop}\cr\diagup\cr\egroup\egroup}%
}\limits}
\newtheorem{theorem}{Theorem}
\newtheorem{lemma}[theorem]{Lemma}
\newtheorem{corollary}[theorem]{Corollary}

\newtheorem{definition}[theorem]{Definition}
\newtheorem{fact}[theorem]{Fact}

\theoremstyle{remark}

\newtheorem{question}[theorem]{Question}

\newcommand{\cx}{\mathbb{C}}

\newcommand{\reals}{\mathbb{R}}
\newcommand{\euc}{\mathbb{E}}

\newcommand\calT{\mathcal T}

\begin{document}

\title{Simplices and spectra of graphs, continued}

\author{Bojan Mohar}
\address{Department of Mathematics, Simon Fraser University, Burnaby, B.C. V5A~1S6}
\email{mohar@sfu.ca}
\thanks{Supported in part by the ARRS Research Program P1-0297, 
by an NSERC Discovery Grant and by the Canada Research Chair program. 
On leave from Department of Mathematics, IMFM \& FMF, University of Ljubljana,
Ljubljana, Slovenia}
\author{Igor Rivin}
\address{Department of Mathematics, Temple University, Philadelphia}
\email{rivin@temple.edu}
\thanks{The author would like to thank the American Institute of Mathematics for an invitation to the workshop on ``Rigidity and Polyhedral Combinatorics", where this work was started. The author has profited from discussions with Igor Pak, Ezra Miller, and Bob Connelly.}

\date{\today} 


\begin{abstract}
In this note we show that the $(n-2)$-dimensional volumes of codimension $2$ faces of an $n$-dimensional simplex are algebraically independent quantities of the volumes of its edge-lengths. The proof involves computation of the eigenvalues of Kneser graphs.
\end{abstract}

\maketitle

\section*{Introduction}

Let $\calT_n$ be the set of congruence classes of $n$-simplices in Euclidean space $\euc^n.$ The set $\calT_n$ is an open manifold (also a semi-algebraic set) of dimension $\binom{n+1}{2}$. Coincidentally, a simplex $T\in \calT_n$ is determined by the $\binom{n+1}{2}$ lengths of its edges. Furthermore, the square of the volume of $T \in \calT_n$ is a polynomial in the squares of the edge-lengths 
$\ell_{ij}=\Vert v_i-v_j\Vert_2$ ($1\le i<j\le n+1$), where $v_1,\dots v_{n+1}$ are the vertices of $T$. This polynomial is given by the Cayley-Menger determinant formula
(cf., e.g., \cite{Blumenthal} or \cite{Berger}):
\begin{equation}
V^2(T) = \dfrac{(-1)^{n+1}}{2^n (n!)^2} \det C\, ,
\label{eq:1}
\end{equation}
where $C$ is the \emph{Cayley-Menger matrix} of dimension $(n+2)\times(n+2)$, whose rows and columns are indexed by $\{0,1,\dots,n+1\}$ and whose entries are defined as follows:

\[
C_{ij} = \begin{cases}
0, & i=j\\
1, & \mbox{if $i=0$ or $j=0$, and $i\ne j$}\\
\ell_{ij}^2\, ,& \mbox{otherwise.}
\end{cases}
\]

Note that an $n$-simplex has $\binom{n+1}{2}$ edges and the same number of 
$(n-2)$-dimensional faces, and so the following question is natural:

\begin{question}
\label{warrenq}
Is the congruence class of every $n$-simplex determined by the $(n-2)$-dimensional volumes of its $(n-2)$-faces?
\end{question}

Question \ref{warrenq} must be classical, but the earliest reference stating it that we are aware of is Warren Smith's PhD thesis \cite{warrenthesis}.

At the AIM workshop on Rigidity and Polyhedral Combinatorics, Bob Connelly (who was unaware of the reference \cite{warrenthesis}) raised the following:

\begin{question}
\label{connellyq}
Is the \emph{volume} of every $n$-simplex determined by the $(n-2)$-dimensional volumes of its $(n-2)$-faces?
\end{question}
In fact, Connelly stated Question \ref{connellyq} for $n=4,$ which is the first case where the question is open. For $n=3$ the answer is trivially "Yes", since $3-2=1,$ and we are simply asking if the volume of the simplex is determined by its edge-lengths. In dimension 2, the answer is trivially "No", since $2-2=0,$ and the volume of codimension-2 faces of a triangle carries no information.

Clearly, the affirmative answer to Question \ref{warrenq} would imply an affirmative answer to Question \ref{connellyq}. In this paper we first show that the answer to Question \ref{connellyq}, and hence also to Question \ref{warrenq} is negative for every $n\ge 4$. We actually found out that this has been answered previously for $n=4$ in \cite{Barrett}, where an example is given and attributed to Philip Tuckey; see also \cite{BianchiModesto}.

Our examples are given in a separate section. Several reasons suggest that the following question may still have an affirmative answer:

\begin{question}
\label{weaker question}
Is it true that for every choice of $\binom{n+1}{2}$ positive real numbers, there are only finitely many congruence classes of 
$n$-simplices whose $(n-2)$-dimensional volumes of the $(n-2)$-faces are equal to these numbers?
\end{question}

In this note we show that a weaker statement holds:

\begin{theorem}
\label{mainthm}
The $\binom{n+1}{2}$ $(n-2)$-dimensional volumes of the $(n-2)$-faces of 
an $n$-simplex are algebraically independent over 
$\cx[\ell_{ij};\, 1\le i<j\le n+1]$.
\end{theorem}

Theorem \ref{mainthm} is clearly a necessary step in the direction of resolving Question \ref{weaker question}, but is far from sufficient. To show it, consider the map of
$\reals^{(n+1)n/2}$ to $\reals^{(n+1)n/2},$ which sends the vector $\ell$ of edge-lengths of an $n$-simplex to the vector $Y$ of volumes of $(n-2)$-dimensional faces. To show Theorem \ref{mainthm}, it is enough to check that the Jacobian $J(\ell)=\partial Y/\partial \ell$ is non-singular at \emph{one} point. We will use the most obvious point $p_1$, the one corresponding to a regular simplex with all edge-lengths equal to $1$. By symmetry considerations, the Jacobian $J(p_1)$ can be written as
$J(p_1) = c M$, where $c$ is a constant and $M$ is
\[
M_{e,F} = 
\begin{cases}
1,& \mbox{if the edge $e$ is incident with the $(n-2)$-face $F$}\\
0,& \mbox{otherwise.}
\end{cases}
\]
The first observation is that the constant $c$ above is not equal to $0$:

\begin{lemma}
$J(p_1) = \dfrac{1}{(n-2)! \, (n-1)^{1/2} \, 2^{(n-4)/2}} \, M$.
\end{lemma}

\begin{proof}
Let $\nu=\dfrac{(n-1)^{1/2}}{(n-2)! \, 2^{(n-2)/2}}$ denote the 
$(n-2)$-dimensional volume of the regular $(n-2)$-simplex with all edge-lengths 1. Let us observe that the volume of a $k$-dimensional simplex is a \emph{homogeneous} function of degree $k$ of the edge-lengths. An application of Euler's Homogeneous Function Theorem shows that at $p_1$, 
\[
\dfrac{\partial Y_F}{\partial \ell_e} = \begin{cases} \frac2{n-1}\,\nu,& \mbox{if the edge $e$ is incident with the $(n-2)$-face $F$}\\
0,& \mbox{otherwise}.
\end{cases}
\]
This implies that $c=\frac{2}{n-1}\,\nu$ and completes the proof.
\end{proof}

\section*{The eigenvalues of $M$}

As shown above, Theorem \ref{mainthm} reduces to the assertion that the determinant of the matrix $M$ is not zero. We will actually be able to compute all eigenvalues of $M$, which is of interest in its own right.

\begin{theorem}
\label{singvals}
Eigenvalues of\/ $M$ are $\lambda_1 = \binom{n-1}{2}$ (simple eigenvalue), 
$\lambda_2 = 1$ with multiplicity $\tfrac{1}{2}(n+1)(n-2)$, 
and $\lambda_3 = 2-n$ with multiplicity $n$.
\end{theorem}

\begin{corollary}
\label{thedet}
The absolute value of the determinant of $M$ equals
\[
   \tfrac12(n-2)^{n+1}(n-1) \neq 0,
\]
for $n>2$.
\end{corollary}

To prove Theorem \ref{singvals}, let us first observe that the
$\binom{n+1}{2}$ rows of $M$ are indexed by the 2-element subsets of
the set $R=\{1,\dots,n+1\}$, and its columns are indexed by the 
$(n-1)$-subsets $F$ of $R$. By replacing each column index $F$ with
its complement $R\setminus F$, then the columns are indexed by the same set
as the rows. After this convention, the matrix $M$ becomes a symmetric
matrix with zero diagonal since $M_{e,f}=1$ if and only if 
$e\subseteq R\setminus f$, which is equivalent to $f\subseteq R\setminus e$.
Therefore, $M$ is the adjacency matrix of a graph $G_n$ whose vertices
are the 2-element subsets of $R$, and two of them are adjacent if and only 
if they are disjoint. Thus, the complement $\overline{G}_n$ of $G_n$ is 
isomorphic to the line graph $L(K_{n+1})$ of the complete graph $K_{n+1}$ 
on $n+1$ vertices.

The eigenvalues of $L(K_{n+1})$ are (see \cite[p.~19]{Biggs}): $t_1 = 2n-2$, 
$t_2 = -2$, and $t_3 = n-3$, with the same multiplicities (respectively) as 
claimed above for the eigenvalues of $M$. Since the graph $L(K_{n+1})$ is
regular, it is an easy exercise to see that its adjacency matrix $A$ and 
the adjacency matrix $M$ of its complement have the same set of eigenvectors. 
By using the fact that
$A + M + I = x^t \cdot x$, where $x = (1,\dots,1)^t$ is the eigenvector of 
$A$ and $M$ corresponding to the dominant eigenvalues of these matrices, 
we conclude that the eigenvalues of $M$ are $\lambda_1 = \binom{n+1}{2}-t_1-1$
and $\lambda_i = -t_i-1$ for $i=2,3$ (preserving multiplicities). Thus,
$\lambda_1 = \binom{n+1}{2}-2n+1=\binom{n-1}{2}$,
$\lambda_2 = 1$, and $\lambda_3 = 2-n$, respectively.

\section*{Singular examples}

Let us consider the $n$-simplex in $\reals^n$ with vertices $v_0,v_1,\dots,v_n$ given as follows. The vertex $v_0$ has the first $n-2$ coordinates equal to 
$((n-1)^{1/2}+1)/(2^{1/2}(n-2))$, while its last two coordinates are 0. For $i=1,2,\dots,n-2$, the vertex $v_i$ has $i$th coordinate equal to $2^{-1/2}$ and all other coordinates 0. These $n-1$ vertices form a regular $(n-2)$-simplex contained in $\reals^{n-2}\subset \reals^n$ with all side lengths 1. Let $a= \tfrac{1}{n-1}\sum_{i=0}^{n-2} v_i$ be its barycenter, and let 
$c := \Vert v_0-a \Vert_2$ denote the distance from $a$ to the vertices $v_i$. A short calculation shows that $c^2 = \tfrac{1}{2} - \tfrac{1}{2(n-1)}$. Now, let $v_{n-1}$ be obtained from $a$ by changing its last two coordinates to be real numbers $p$ and $q$ satisfying $p^2 + q^2 = 1-c^2$. Similarly, let $v_n$ be obtained in the same way by choosing another pair $r,s$ of numbers satisfying $r^2+s^2=1-c^2$. This gives rise to an $n$-simplex whose all sides are equal to 1 except for the side $v_{n-1}v_n$ whose square length is $t:=(p-r)^2+(q-s)^2$. By fixing $t$, this simplex is determined up to congruence, and we denote it by  $T(t)$. Observe that $t$ may take any value between 0 and $4(1-c^2)$, by selecting $p,q,r,s$ appropriately.

Next we observe that the volumes of the $(n-2)$-faces of $T(t)$ take only two values. If an $(n-2)$-face does not contain both $v_{n-1}$ and $v_n$, then it is a regular simplex, whose volume is independent of $t$. On the other hand if an $(n-2)$-simplex contains $v_{n-1}$ and $v_n$, its volume $w=w(t)$ is uniquely determined by $t$. In fact, if we put the square distances in the Cayley-Menger determinant, we conclude that $w(t)^2$ is a quadratic polynomial in $t$, 
$w(t)^2 = \alpha t^2 + \beta t + \gamma$. If $t=0$, the volume is 0, so $\gamma = 0$. For $t=1$ we have the regular $(n-2)$-simplex, so 
$\alpha + \beta = \frac{n-1}{2^{n-2}((n-2)!)^2}$.
Finally, using (\ref{eq:1}) (with the value of $n$ being replaced by $n-2$)
and looking at the Cayley-Menger determinant expansion term with $t^2$, we conclude that 
\[
   \alpha = - \frac{(-1)^{n-1}}{2^{n-2}((n-2)!)^2}\, \det(J_{n-2}-I_{n-2}),
\]
where $J_{n-2}$ is the all-1-matrix and $I_{n-2}$ is the identity matrix of order $n-2$. Since $\det(J_{n-2}-I_{n-2}) = (-1)^{n-3}(n-3)$, we conclude that
$\alpha = -(n-3)2^{2-n}/((n-2)!)^2$ and $\beta = (n-2)2^{1-n}/((n-2)!)^2$.
In particular,
\[
   w(t)^2 = \frac{1}{2^{n-2}((n-2)!)^2}\, ((3-n)t^2 + (2n-4)t).
\]
This function is symmetric around the point $t_0=\tfrac{n-2}{n-3}$. Consequently, the non-congruent $n$-simplices $T(t_0-x)$ and
$T(t_0+x)$ have the same $(n-2)$-volumes of their
$(n-2)$-faces for each admissible value of $x$, 
i.e. for $0 < x < \tfrac{n-2-4/(n-1)}{n-3}$. These examples thus show that Questions \ref{warrenq} and \ref{connellyq} have negative answers.

\section*{Concluding remarks}

One can ask the same question as above for other dimension-complementary volumes,
i.e. about the volumes of the $(k-1)$-faces and the $(n-k)$-faces of 
an $n$-simplex, where $2\le k\le n/2$. 
If one would compare, similarly as in the case $k=2$ above,
the dependence of $(n-k)$-volumes of an $(n-k)$-face $Q$ on 
the $(k-1)$-volumes of the $(k-1)$-faces $F\subset Q$, the corresponding
``Jacobian'' would again be a constant multiple of a symmetric matrix $M$, whose entries are indexed by the $k$-subsets of the set $R=\{1,\dots,n+1\}$ (after the column indices pass to the complementary subsets), and
\begin{equation}
\label{eq:Mk}
M_{E,F} = 
\begin{cases}
1,& \mbox{if the $E\cap F=\emptyset$}\\
0,& \mbox{otherwise.}
\end{cases}
\end{equation}
The graph whose adjacency matrix is $M$ is known as the {\em Kneser graph\/}
$K(n+1,k)$. Its eigenvalues can be comuted using the methods from the theory of association schemes and can be found, for example, 
in \cite[Section~9.4]{GodsilRoyle}. 

\begin{theorem}
\label{Kneser}
Let $n$ and $k$ be integers, where $2\le k\le n/2$, and let $M$ be the matrix
of order $\binom{n+1}{k}\times \binom{n+1}{k}$ whose entries are determined
by (\ref{eq:Mk}). The eigenvalues of\/ $M$ are the integers
$$
 \lambda_i = (-1)^i \binom{n-k-i+1}{k-i}, \qquad i=0,1,\dots,k. 
$$
\end{theorem}

Since $2\le k\le n/2$, none of the eigenvalues in Theorem \ref{Kneser} is zero.
This raises the question whether there is an analogy with 
Theorem \ref{mainthm} for $2\le k\le n/2$, between the collection of 
the $\binom{n+1}{k}$ $(n-k)$-dimensional volumes of the $(n-k)$-faces of 
an $n$-simplex and the collection of all $(k-1)$-dimensional volumes of its $(k-1)$-faces.

As a final remark, we would like to point out that our original approach to this problem \cite{RivinArxive} used results about 
{\em divisors} \cite{cvetkovicspectra} 
(also known as {\em equitable partitions}~\cite{GodsilRoyle}) combined with the representation theory of the symmetric group and the notion of {\em Gelfand pairs} 
as developed in \cite{diaconisbook}.

\bibliographystyle{plain}

\end{document}